\documentclass{jssc}

\def\textsubscript#1%
{$_{\text{#1}}$}

\def\cdd{\mbox{\boldmath$\cdot$}~}

\input{amssym.def}

\makeatletter
\def\@oddfoot{\hfill}
\newcount\shumeicount
\def\setshumei#1#2#3{%
  \shumeicount=\count0
  \def\@oddhead{%
    \raise-5pt\hbox to0pt{\vrule width\hsize height 0pt depth 0.4pt\hss}\relax
    \ifnum \shumeicount=\count0
      \raise-7pt\hbox to0pt{\vrule width\hsize height 0pt depth 0.4pt\hss}\relax
      #1
    \else
      \ifodd\count0
        #2
      \else
        #3
       \fi
     \fi
  }%
}
\makeatother
\makeatletter
\def\@oddfoot{\hfill}
\newcount\shujiaocount
\def\setshujiao{%
  \shujiaocount=\count0
  \def\@oddfoot{%
      \ifodd\count0
      \else
      \fi
  }%
}
\makeatother
\def\title#1#2#3#4{{
  \vspace*{0.3cm}
  \begin{flushleft} \Large\bf #1\end{flushleft}
  \vspace*{-0.2cm}
      \begin{flushleft}
      \bf #2
      \end{flushleft}
      \footnotetext{\hspace{-6mm} #3\\ #4}}}

\def\dshm#1#2#3#4
{\setshumei{J Syst Sci Complex (#1) #2:
{\thepage--\pageref{LastPage}}\hfill}
            {\hfill {\small #3}\hfill\hbox to0pt{\hss\thepage}}
            {\hbox to0pt{\thepage\hss}\hfill {\small #4}\hfill
            }
            \setshujiao}
\def\drd#1#2
{{\vskip 1cm\small \begin{flushleft}
 #1 \\
 #2\\
\copyright The Editorial Office of JSSC \&  Springer-Verlag GmbH
Germany 2018
\end{flushleft}}}

\def\tilde{\widetilde}
\def\bar{\overline}
\def\epsilon{\varepsilon}

\makeatletter
\newcommand*\doTRANS[2]{\raisebox{\depth}{$\m@th#1\intercal$}}
\makeatother

\usepackage{color}
\usepackage{subfigure}

\begin{document}
%
\title{Linear Quadratic Mean Field Games with Quantile-Dependent Cost Coefficients$^*$}
{\uppercase{Gao}
Shuang \cdd \uppercase{Malham\'e}  Roland P.}
{\uppercase{Gao} Shuang   \\
Department of Electrical Engineering, Polytechnique Montr\'eal, Montreal, Canada; Group for Research in Decision Analysis (GERAD), Montreal, Canada.  Email: shuang.gao@polymtl.ca.  \\
\uppercase{Malham\'e} Roland P.  \\
Department of Electrical Engineering, Polytechnique Montr\'eal, Montreal, Canada; Group for Research in Decision Analysis (GERAD), Montreal, Canada.  Email: roland.malhame@polymtl.ca   
   } 
{$^*$This research was supported by NSERC (Canada). \\ 
}
%
%
%
\dshm{20XX}{XX}{Linear Quadratic Mean Field Games with Quantile-Dependent Cost}{\uppercase{Gao Shuang} $\cdd$ \uppercase{Malham\'e  Roland P.} 
}
%
\Abstract{This paper studies a class of linear quadratic mean field games where the coefficients of quadratic cost functions depend on both the mean and the variance of the population's state distribution through its  quantile function. 
Such a formulation allows for modelling agents that are sensitive to not only the population average but also the population variance.
The corresponding mean field game equilibrium is identified, which involves solving two coupled differential equations: one is a Riccati equation and the other the variance evolution equation. Furthermore, the conditions for the existence and uniqueness of the mean field equilibrium are established.  Finally, numerical results are presented to illustrate the behavior of two coupled differential equations and the performance of the mean field game solution.}      




\section{Introduction}

  Mean Field Game (MFG) theory addresses the complexity of modelling and analyzing large populations of strategically interacting dynamic agents  \cite{HMC06, lasry2006jeux1} by providing an  framework to identify decentralized approximate solutions for large-population stochastic games. 
  MFG finds applications in various domains such as finance (e.g., algorithmic trading \cite{casgrain2020mean}), energy systems \cite{kizilkale2019integral}, transportation dynamics \cite{huang2021dynamic}, and epidemic modelling \cite{aurell2022optimal}, among others. 
 Many variations and extensions of MFG are now available, see for instance \cite{subramanian2019reinforcement, lauriere2022scalable, tembine2013risk, bauso2016robust, kizilkale2012mean} where techniques from adaptive control, reinforcement learning, risk-sensitive control, and robust control have been adapted to MFG problems.
One  variation of MFG problems incorporates the quantile of the mean field distribution (see \cite{tembine2017quantile,tchuendom2019quantilized,RinelQuantile2024Automatica}). The quantile function allows a simple representation of  the mean-field distribution involving  higher order moments (e.g., variances in the Gaussian case) and the flexibility of adjusting the distribution couplings via the quantile parameter \cite{tchuendom2019quantilized,RinelQuantile2024Automatica}.

  In this work, we will investigate linear quadratic MFGs where the cost coefficients depend on the quantile of the mean field distribution.  Due to the Gaussian nature of the problem, one can simplify the analysis, resulting in a linear quadratic MFG where cost coefficients depends on the mean and variance of the mean field distribution. This allows us to characterized the MFG solutions and establish conditions for the existence and uniqueness of solutions.
 
  One closely related work is the integral control based MFG \cite{kizilkale2019integral} where the coefficient in the running cost depends on an integral term related to the population behaviour over time to account for requirements from energy applications. The dependence of the cost coefficient on the population in \cite{kizilkale2019integral} differs from the current paper. Another related work is  \cite{fischer2016continuous} where the mean-variance portfolio optimization problems treated in using a mean field approach,  and hence is different from the current paper in the problem formulation.

The paper is organized as follows: Section \ref{sec:problem} presents the problem formulation. Then  the MFG solution is established in Section \ref{sec:MFG-sol}. A special case where the cost coefficient does not depend on the mean is treated in Section \ref{sec:special-case}. 
 Conditions for the existence and uniqueness of solutions are established in Section \ref{sec:exist-unique}. The properties of the solutions are numerically illustrated in Section \ref{sec:numerics}. Finally, future generalizations are discussed in Section \ref{sec:conclusion}.

\section{Problem Statement} \label{sec:problem}

A quantile of a random variable \(X\) (or its distribution $F_X$) provides a {least upper bound}  \(x\) such that a certain {fraction} \(\alpha\) (with \(0 \leq \alpha \leq 1\)) of the probability mass {is guaranteed to lie below or at} \(x\). 
{More precisely,} the \emph{\(\alpha\)-quantile} of a random variable $X$ (resp. its distribution $F_X$) is denoted by \(\mathcal{Q}_X(\alpha)\) (resp. $\mathcal{Q}_{F_X}(\alpha)$) and is defined by
\begin{equation}\label{eq:quantile}
	\mathcal{Q}_X(\alpha) =\mathcal{Q}_{F_X}(\alpha)  := \inf \{ x \in \mathbb{R} : F_X(x) \geq \alpha \}, \quad \alpha \in [0,1],
\end{equation}
where \(F_X(x)\) denotes the cumulative distribution function (CDF) of \(X\). %
For the special case of a Gaussian random variable \( X \sim \mathcal{N}(\mu, \sigma)\) with the mean $\mu$ and the standard deviation $\sigma$, the quantile function satisfies
$
\mathcal{Q}_X(\alpha) = \mu + \sigma \Phi^{-1}_{\alpha}
$
where \( \Phi^{-1}_{(\cdot)} \) is the probit function, i.e., the inverse CDF of the standard (scalar-valued) normal distribution (see Figure~\ref{fig:probit}).
\begin{figure} 
\centering
	\includegraphics[width=5cm]{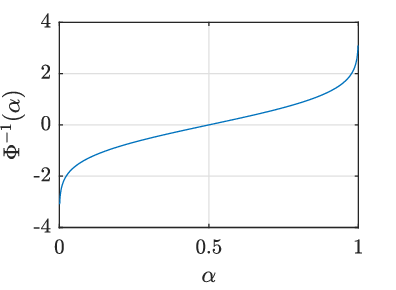}
	\caption{The probit function for the standard (scalar-valued) Gaussian distribution} \label{fig:probit}
\end{figure}

Consider the following linear quadratic stochastic game with $n$ agents: the dynamics of agent $i$ are given by
\begin{equation}\label{eq:ind-dyn}
	dx_i(t) = (ax_i(t)+bu_i(t))dt +\sigma dw_i(t), \quad x_i(t), u_i(t), w_i(t) \in \mathbb{R}, \quad  t \in [0,T],
\end{equation}
and the cost of agent $i$ is given by
\begin{equation}\label{eq:ind-cost}
	J_i(u_i, u_{-i}) = \mathbb{E} \int_0^T \left( \|x_i(t)-\bar{x}^n(t)\|_{{q}_t^{\alpha}(i,n) }^2 + \|u_i(t)\|_r^2 \right)dt
\end{equation}
where  $a\in \mathbb{R}$, $b\neq 0$, $r>0$,  $x_i(t)$ (resp. $u_i(t)$) is the state (resp. the control action) of agent $i$, $u_i$ denotes the strategy of agent $i$, $u_{-i}$ denotes the strategies of all other agents except agent~$i$, $w_i(t)$ is the standard Brownian motion, the initial $x_i(0)\sim \mathcal{N}(\mu_0, V_0)$ is normally distributed with variance $V_0\geq 0$, and the average of the population  state is denoted by $\bar{x}^n(t) := \frac{1}{n} \sum_{i=1}^n x_i(t)$. Let   the empirical state distribution for all agents except agent $i$ be denoted by $\mathcal{D}_{t}^{n,-i}(z) : = \frac{1}{n-1} \sum_{j=1, j\neq i}^n \mathbf{1}_{(x_i(t)\leq z)}$  %
 with $\mathbf{1}_{(\cdot)}$ as the indicator function. The penalty coefficient ${q}_t^{\alpha}(i,n) $ in the running cost  is given by
\begin{equation}
	{q}_t^{\alpha}(i,n) : = q\left(1+e^{\mathcal{Q}_{\mathcal{D}_{t}^{n,-i} }(\alpha)}\right), \quad ~\text{with}~ \alpha  \in (0,1),~ q \geq 0,
\end{equation}
 where the $\alpha$-quantile of the empirical distribution following the definition in \eqref{eq:quantile}  is given by
\[
\mathcal{Q}_{\mathcal{D}_{t}^{n,-i} }(\alpha) = \inf \{ z \in \mathbb{R} : \mathcal{D}_t^{n,-i} (z) \geq \alpha \}.
\]
 For a higher quantile (with a larger quantile parameter $\alpha $), the penalty for deviating from the population average $\bar{x}^n$ is larger and hence agents are encouraged to follow more closely the population average $\bar{x}^n$.
 
{To identify approximate decentralized solutions for the problem above when the population size is large, we apply the mean field game approach \cite{HCM07} by taking the population limit and obtain approximate solutions via solving the limiting game problem.}
In the mean field limit, we obtain the corresponding LQG mean field tracking problem with the dynamics 
\begin{equation}\label{eq:dyn}
	dx_i(t) = (ax_i(t)+bu_i(t))dt +\sigma dw_i(t), \quad x_i(t), u_i(t), w_i(t) \in \mathbb{R}, \quad  t \in [0,T],
\end{equation}
and the cost 
\begin{equation}\label{eq:cost}
	J_i(u_i, \bar{x}) = \mathbb{E} \int_0^T \left( \|x_i(t)-\bar{x}(t)\|_{{q}_t^{\alpha}}^2 + \|u_i(t)\|_r^2 \right)dt
\end{equation}
where the mean state (if exists) is given by $$\bar{x}(t) = \lim_{n\to \infty} \frac{1}{n} \sum_{i=1}^n x_i(t)$$ and the penalty coefficient in the running cost is given by
\begin{equation}
	{q}_t^{\alpha} : = q (1+e^{\mathcal{Q}_{\mathcal{D}_t^\infty}(\alpha)}), \quad ~\text{with}~ \alpha 
	\in (0,1), ~ ~ q \geq 0
\end{equation}
with $\mathcal{D}_t^\infty$ denoting the distribution of the infinite population states at time $t$ and $q_t^\alpha$ (with an abuse of notation) representing the cost coefficient in the population limit. We note that $\mathcal{D}_t^\infty$ and $q_t^\alpha$  do not depend on the index $i$ as individual agent state has a negligible impact on the limit distribution within the infinite population.
 With a higher quantile via a larger quantile parameter $\alpha $,  agents are encouraged to follow more closely the mean state $\bar{x}$.

\section{Mean Field Game Solution} \label{sec:MFG-sol}

We follow the standard idea in \cite{HCM07} to derive the mean field game solutions. First, assuming $(\bar{x}(t))_{t\in [0,T]}$ is given, the optimal tracking solution via Dynamic Programming is given by  
\begin{equation}\label{eq:feeback-form}
	\begin{aligned}
		 u_i^*(t) &= - \frac{b}{r}(\Pi_t x_i(t)+ s_t)
		 	\end{aligned}
\end{equation}
where 
\begin{equation}
	\begin{aligned}
		-\dot{\Pi}_t &= 2 a \Pi_t     -  \frac{b^2}{r} \Pi_t^2 +  q_t^\alpha, \quad \Pi_t\geq 0, \quad   \Pi_T =   0,\\
	-\dot{s}_t &= \big(a  -  \frac{b^2}{r} \Pi_t \big) {s}_t - q_t^\alpha\bar{x}(t), \qquad s_T = 0.
	\end{aligned}
\end{equation}
The control law above corresponds to the best response for the representative agent $i$ when the mean trajectory $(\bar{x}(t))_{t\in [0,T]}$ is given.  
Then the closed-loop dynamics for a representative agent~$i$ that implements the best response are given by 
\[
dx_i (t) = (a - \frac{b^2}{r}\Pi_t) x_i(t) dt -  \frac{b^2}{r} s_t dt +\sigma dw_i(t)
\]
and if all the agents follows the best response, the evolution of the  mean field, if exists, satisfies
\[
d{\bar{x}}(t) = \left((a - \frac{b^2}{r}\Pi_t) \bar{x}(t)-  \frac{b^2}{r} s_t \right)dt
\]
with  $ \bar{x}(t)= \mathbb{E}{x}_i(t)$.

Let $e_i(t) = x_i(t) -\bar{x}(t)$ and then the evolution $e_i(t)$ satisfies
\[
de_i(t) = (a - \frac{b^2}{r}\Pi_t) e_i(t)dt +\sigma dw_i(t).
\]
Applying It\^o's lemma to $ e_i^2$ yields the following
\[
\begin{aligned}
	de_i^2  & = 2e_i de_i  + \sigma^2 dt \\
	& = 2(a - \frac{b^2}{r}\Pi_t) e_i^2dt + 2\sigma e_i dw_i  + \sigma^2 dt.
\end{aligned}
\]
Let \( V(t) := \mathbb{E}[e_i^2 ] \) denote the variance of the state \( x_i(t) \) of a representative agent $i$.
Then taking the expectation on both sides (of the corresponding integral equation)  yields the dynamics for the variance \( V(t) \) as follows: %
\[
\begin{aligned}
	\dot{V}(t) = ~ & 2(a - \frac{b^2}{r}\Pi_t) V(t)  + \sigma^2
\end{aligned}
\]
with the initial condition (given by the variance of the initial condition of the generic agent) $V(0)=V_0:=\mathbb{E}(x_i(0)-\mathbb{E}{x_i(0)})^2$. See e.g., \cite{aastrom2012introduction} for  the covariance evolution for linear stochastic differential equations. 

We note that the control follows a linear state-feedback form as specified in \eqref{eq:feeback-form}, and this implies that the state will be Gaussian distributed and hence the limit population state distribution $\mathcal{D}_t^\infty$ is in fact $ \mathcal{N}(\bar{x}(t), \sqrt{V(t)})$.  Thus the quantile can  {be expressed in terms of} the mean and the variance, resulting in 
\begin{equation}\label{eq:constraint}
	{q}_t^{\alpha} : = q (1+e^{\mathcal{Q}_{\mathcal{D}_t^\infty}(\alpha)}) =  q (1+e^{\bar{x}(t) + \sqrt{V(t)} \Phi^{-1}_{\alpha}}), ~\text{with}~ \alpha \in(0,1).
\end{equation}
 Therefore we obtain the following result.
\begin{lemma}[MFG Solutions]\label{lem:mfg-sol}
	The solution to the mean field game problem defined by \eqref{eq:dyn}, \eqref{eq:cost} and \eqref{eq:constraint}, if exists, is given by 
\begin{equation}
	\begin{aligned}
		 u_i(t) &= - ~\frac{b}{r}(\Pi_t x_i(t)+ s_t)
		 	\end{aligned}
\end{equation}
where 
\begin{align}
		-\dot{\Pi}_t &= 2 a \Pi_t     -  \frac{b^2}{r}\Pi_t^2  +  q_t^\alpha, \quad \Pi_t\geq 0,  \quad  \Pi_T =   0 \label{eq:dyn-Pi}\\
	-\dot{s}_t &= \big( a  -  \frac{b^2}{r} \Pi_t \big) {s}_t - q_t^\alpha\bar{x}(t), \quad s_T = 0\label{eq:dyn-st}\\
	\dot{\bar{x}}(t) & = (a - \frac{b^2}{r}\Pi_t) \bar{x}(t)-  \frac{b^2}{r} s_t, \quad \bar{x}(0)= \mu_0 \label{eq:dyn-barx}\\
	\dot{V}(t)   &= 2(a - \frac{b^2}{r}\Pi_t) V(t)+ \sigma^2, \quad V(0)= V_0 \label{eq:dyn-V}
\end{align}
with the following constraint 
\begin{equation}
	{q}_t^{\alpha}  = q (1+e^{\bar{x}(t) + \sqrt{V(t)} \Phi^{-1}_{\alpha}}), ~\alpha \in(0,1), ~ q\geq 0.
\end{equation}
\end{lemma}
The solution can be further simplified by investigating the dynamics for $\bar{x}$. If we take the second order derivative for $\bar{x}$ from \eqref{eq:dyn-barx}, we obtain 
\[
\begin{aligned}
	{\ddot{\bar{x}}}(t) & =  - \frac{b^2}{r}\dot{\Pi}_t \bar{x}(t) + (a - \frac{b^2}{r}\Pi_t) \dot{\bar{x}}(t)  -  \frac{b^2}{r} \dot{s_t}(t). 
\end{aligned}
\]
Then substituting the derivatives above using \eqref{eq:dyn-Pi},  \eqref{eq:dyn-st} and \eqref{eq:dyn-barx} yields the following  dynamics
\[
	{\ddot{\bar{x}}} (t) 	= a^2 \bar{x} (t) 
\]
with the initial condition $\bar{x}(0)= \mu_0$.
If $a\neq 0 $, the solution is then given by the following form
$
\bar{x}(t) = C_1 e^{at} + (\mu_0 - C_1) e^{-at}
$
where $C_1$ is the constant to be determined based on the boundary conditions for $s$ and $\Pi$ as follows: 
from \eqref{eq:dyn-barx}
we obtain (with $b\neq 0$)
\[
\begin{aligned}
s_t  & =  \frac{r}{b^2} \left[(a - \frac{b^2}{r}\Pi_t) \bar{x}- 	\dot{\bar{x}}\right]; 
\end{aligned}
\]
then applying the terminal conditions $s_T=0$ and $\Pi_T =0$ yields 
$
0  =  \frac{r}{b^2} \left[a \bar{x}(T)- 	\dot{\bar{x}}(T)\right], 
$ and hence 
$
a C_1 e^{aT} + a (\mu_0 - C_1) e^{-aT} = C_1 a e^{aT}-  (\mu_0 - C_1) a e^{-aT}
$
implying  $C_1 = \mu_0$. 
Thus the solution for $\bar{x}$ is given by 
\begin{equation}\label{eq:barx-exp}
	\bar{x} (t) = e^{at} \mu_0, \quad \forall t \in [0,T].
\end{equation}
It is easy to verify that for the case $a=0$, the solution \eqref{eq:barx-exp} also holds.

The feature that the mean does not depend on $q^{\alpha}_t$, $\Pi_t$ or $s_t$  allows us to have the simplified solution below. 

\begin{lemma}[Simplified MFG Solutions]\label{lem:simplified-mfg-sol}
	The solution to the mean field game problem defined by \eqref{eq:dyn}, \eqref{eq:cost} and \eqref{eq:constraint}, if exists, is given by 
\begin{equation}
	\begin{aligned}
		 u_i(t) &= - ~\frac{b}{r}(\Pi_t x_i(t)+ s_t)
		 	\end{aligned}
\end{equation}
where $(\Pi_{t})_{t\in[0,T]}$ solves
\begin{align}
		-\dot{\Pi}_t &= 2 a \Pi_t     -  \frac{b^2}{r}\Pi_t^2  +  q_t^\alpha, \quad \Pi_t\geq 0,  \quad  \Pi_T =   0 \label{eq:dyn-Pimv}\\
	\dot{V}(t)  & = 2(a - \frac{b^2}{r}\Pi_t) V(t)+ \sigma^2, \quad V(0)= V_0 \label{eq:dyn-Vmv}
\end{align}
and $(s_{t})_{t\in[0,T]}$ solves
\begin{align}
			-\dot{s}_t &= \big( a  -  \frac{b^2}{r} \Pi_t \big) {s}_t - q_t^\alpha\bar{x}(t), \quad s_T = 0
\end{align}
with 	  
\begin{equation} \label{eq:Q-constraint}
	{q}_t^{\alpha} = q (1+e^{\bar{x}(t) + \sqrt{V(t)} \Phi^{-1}_{\alpha}})
\end{equation}
 and ${\bar{x}}(t)  = e^{at} \mu_0$.
\end{lemma}

Clearly, the existence and uniqueness of solutions depend on the existence of the solution pair  $(\Pi_{t}, V(t))_{t\in[0,T]}$ to the coupled equations \eqref{eq:dyn-Pimv} and \eqref{eq:dyn-Vmv} with the coupling \eqref{eq:Q-constraint}. 

\section{Special Case with Variance Dependent Costs} \label{sec:special-case}

If the cost coefficient is defined by 
\begin{equation}\label{eq:variance-dependent-case}
	{q}_t^{\alpha} = q (1+e^{ \sqrt{V(t)} \Phi^{-1}_{\alpha}})
\end{equation}
 such that the mean does not appear in $q^{\alpha}_t$, then one can decouple the dynamics of $s(t)$ and $\bar{x}(t)$ via a decoupling Riccati equation (see e.g., \cite{huang2012social,bensoussan2016linear,salhab2016collective}). 
If we introduce 
$
s_t = P_t \bar{x}(t),
$
then
\begin{align}\label{eq:decoupling-pro}
	\dot{s}_t = \dot{P}_t \bar{x}(t)  + P_t \dot{\bar{x}}(t). 
\end{align}
Replacing $\dot{s}_t$ and $\dot{\bar{x}}(t)$ using \eqref{eq:dyn-barx} and \eqref{eq:dyn-st} yields
\[
\begin{aligned}
	-\dot{P}_t \bar{x}  
	&  = \bigg[ 2(a  -  \frac{b^2}{r} \Pi_t) P_t -  \frac{b^2}{r} P_t^2   - q_t^\alpha \bigg] \bar{x}. 
\end{aligned}
\]
Since this should hold for all $\bar{x}$, we obtain the following  Riccati equation that decouples the dynamics for $\bar{x}$ and $s$:
\[
\begin{aligned}
	-\dot{P}_t  = & ~ 2(a  -  \frac{b^2}{r} \Pi_t) P_t  -  \frac{b^2}{r} P_t^2   - q_t^\alpha,  \quad P_T =0 .
\end{aligned}
\]

\begin{remark}
Such an analysis for decoupling $s$ and $\bar{x}$ would not work  when $q_t^\alpha$ depends on the mean $\bar{x}$. This is because $P(\bar{x},t)$ depends on $\bar{x}$ through the coefficient $ q_t^\alpha$ of the Riccati equation for $P$ and in this case one  should introduce 
$
s_t = P(\bar{x}(t),t) \bar{x}(t).
$
Hence one should use 
\[
\begin{aligned}
	\dot{s}_t  =  \bigg[\frac{\partial }{\partial t} \dot{P}   + \frac{\partial P }{\partial{\bar{x}} } \cdot  \dot{\bar{x}}(t) \bigg]\bar{x}(t) +  P\dot{\bar{x}}(t)
\end{aligned}
\] to replace \eqref{eq:decoupling-pro}, 
and it may not lead to a decoupling Riccati equation for $P$.
\end{remark}
Thus, the solution for the  MFG problem defined by \eqref{eq:dyn}, \eqref{eq:cost} and \eqref{eq:variance-dependent-case},   if exists is given by
\begin{equation}
	\begin{aligned}
		 u_i(t) &= -  \frac{b}{r}(\Pi_t x_i(t)+ P_t \bar{x})\\
		 	\end{aligned}
\end{equation}
with $\Pi_t$ and $V(t)$ given by \eqref{eq:dyn-Pi} and \eqref{eq:dyn-V}, and $\bar{x}$ and $P_t$ given by
\begin{align}
	\dot{\bar{x}} & = (a -  \frac{b^2}{r}(\Pi_t+P_t)    ) \bar{x} , \quad \bar{x}(0)= \mu_0\\
			-\dot{P}_t &  =   2(a  -  \frac{b^2}{r} \Pi_t) P_t -  \frac{b^2}{r} P_t^2   - q_t^\alpha,  \quad P_T =0  \label{eq:Pt}
	\end{align}
with $
	q_t^\alpha = q (1+e^{ \sqrt{V(t)} \Phi^{-1}_{\alpha}})\geq 0$.
	We note that  the last equation \eqref{eq:Pt} can be solved after the two equations \eqref{eq:dyn-Pi} and \eqref{eq:dyn-V} for $\Pi_t$ and $V(t)$  are solved. 
\begin{remark}
The solution above has a special structure that leads to further simplifications.  
If we denote 
$
H_t = \Pi_t + P_t,
$
then summing up both sides of the dynamics for $\Pi$ and $P$ yields the dynamics for $H$ as follows:
\begin{equation}\label{eq:dyn-H}
	- \dot{H}_t= 2a H_t - \frac{b^2}{r}(H_t)^2, 
\end{equation}
with the terminal condition $ H_T =0.$
Clearly the solution is  $H_t=0$ for all $t\in [0,T]$.
Then the MFG solution can be replaced by 
\begin{equation}
	\begin{aligned}
		 u_i^*(t) &= -  \frac{b}{r}\Big(\Pi_t x_i(t)+ (H_t-\Pi_t) \bar{x}(t)\Big) = -\frac{b}{r}\Pi_t (x_i(t)-\bar{x}(t))
		 	\end{aligned}
\end{equation}
where the control gains are given by 
	\begin{align}
			-\dot{\Pi}_t &= 2 a \Pi_t     -  \frac{b^2}{r}\Pi_t^2  +  q_t^\alpha, \quad  \Pi_T =   0  \label{eq:Pi-ODE}\\
				\dot{V}(t)   &= 2(a - \frac{b^2}{r}\Pi_t) V(t)+ \sigma^2, \quad V(0)= V_0 \label{eq:V-ODE}
	\end{align} 
	with 
	$q_t^\alpha = q (1+e^{ \sqrt{V(t)} \Phi^{-1}_{\alpha}})\geq 0$,
	and the mean field dynamics are given by 
\begin{equation}
\begin{aligned}
	\dot{\bar{x}} & = (a -  \frac{b^2}{r} H_t    ) \bar{x} = a\bar{x} , \quad \bar{x}(0)= \mu_0,
\end{aligned}
\end{equation}
consistent with the solution in \eqref{eq:barx-exp}. 
Such a decomposition  of the control gains can be viewed as a special case of the decomposition for  network-coupled control problems detailed in  
\cite{ShuangAdityaTCNS20} using a spectral decoupling approach.

	However, in general for mean field game problems with mean field couplings in the dynamics, 
	the solution for $H_t := P_t + \Pi_t$ may not necessarily be zero since the boundary condition could be nonzero or the zero order term in the dynamics for $H_t$ could be non-zero, and for such cases an extra condition is needed for the time horizon to account for the  finite escape time. In addition, in cases where the state dimension is higher than $1$, one cannot have the structure $H_t = P_t + 
\Pi_t$  as the Riccati equation for $\Pi_t$ will be  a non-symmetric Riccati equation (see e.g., \cite{huang2019linear,ShuangPeterMinyiTAC21}).  
\end{remark}
 \section{Existence and Uniqueness of Solutions}\label{sec:exist-unique}

	In this section, we  establish conditions under which the two-point boundary value problem specified by \eqref{eq:dyn-Pimv} and \eqref{eq:dyn-Vmv} has a solution. 
\begin{proposition}[Existence]
 	Consider $[0,T]$ and $K  : = \{\Pi \in C([0, T], \mathbb{R}): \| \Pi\|_{\infty} \leq M\}$. If $T$ and $M$ satisfy the relation 
 	\begin{equation}\label{eq:T-M relation}
	T\left[2|a| M  + \frac{b^2}{r} M^2+ q  +  q \exp\left[ \mu^* + |\Phi^{-1}_{\alpha}| (V_0 +\sigma^2 { T } )^{\frac12} \exp \left( T \big(|a| + \frac{b^2}{r} M \big)  \right)  \right] \right]		 \leq M	,
\end{equation}
with $\mu^*  : = \max\{\mu_0, e^{aT}\mu_0\}$,
 	then the coupled differential equations \eqref{eq:dyn-Pimv} and \eqref{eq:dyn-Vmv} with the coupling constraint \eqref{eq:Q-constraint} have at least one joint solution $(V, \Pi)$ with $\Pi \in K$ and $V \in C([0,T], \mathbb{R})$. 
 \end{proposition} 
 \begin{proof}
The solution to \eqref{eq:dyn-Vmv} written in the integral form is given by
\[
V(t) = \phi(t, 0)  V_0  + \sigma^2 \int_0^t \phi(t,\tau)  d\tau  
\]
where $ \phi(t,\tau) : = \text{exp} \left(\int_\tau^t 2(a-\frac{b^2}{r} \Pi_s)ds \right)$. 
Substituting $V(t)$ in \eqref{eq:dyn-Pimv}  yields the following 
\begin{equation} \label{eq:Pi}
	\begin{aligned}
	-\dot{\Pi}_t &= 2 a \Pi_t     -  \frac{b^2}{r}\Pi_t^2  + q+  q \text{exp}\bigg [ \bar{x}(t)+  \Phi^{-1}_{\alpha}\big( \phi(t, 0)  V_0  + \sigma^2 \int_0^t \phi(t,\tau)  d\tau\big)^{\frac12}  \bigg], \quad  \Pi_T =   0.
\end{aligned}
\end{equation}
Since \(\Pi_T = 0\), integrating both sides  from \(t\) to \(T\) yields
\begin{equation}
	\begin{aligned}
 &\Pi_t  = \int_t^T \bigg(2a \Pi_s - \frac{b^2}{r} \Pi_s^2 + q \\&  +  q\exp\bigg[\bar{x}(s)+ \Phi^{-1}_{\alpha}  \bigg(\text{exp}(\int_0^s 2(a-\frac{b^2}{r} \Pi_u)du)  V_0 + \sigma^2 \int_0^s \exp \left(\int_\tau^s 2(a - \frac{b^2}{r} \Pi_u ) du \right) d\tau \bigg)^{\frac12} \bigg]\bigg) ds.	
\end{aligned}
\end{equation}
Let  \(\mathcal{T}\) denote the operator on the Banach space of continuous functions  \(C([0, T], \mathbb{R})\) endowed with the uniform norm $\|\cdot\|_\infty$ (i.e., $\|x\|_\infty : = \sup_{t\in [0,T]}|x(t)|$ for $x \in C([0, T], \mathbb{R})$),  defined by the following: for any $\Pi \in C([0, T], \mathbb{R})$,
\begin{equation} \label{eq:Tao-Mapping}
	\begin{aligned}
	& \mathcal{T}(\Pi)(t)   = 
  \int_t^T \bigg(2a \Pi_s - \frac{b^2}{r} \Pi_s^2 + q \\&  +  q\exp\left[\bar{x}(s) + \Phi^{-1}_{\alpha}  \bigg(\text{exp}(\int_0^s 2(a-\frac{b^2}{r} \Pi_s)ds)  V_0 + \sigma^2 \int_0^s \exp \left(\int_\tau^s 2(a - \frac{b^2}{r} \Pi_u ) du \right) d\tau \bigg)^{\frac12} \right]\bigg) ds
\end{aligned}
\end{equation}
with $t \in [0,T]$. 

Let's consider the set of uniformly bounded continuous function over $[0,T]$ given by 
\[
K  : = \{\Pi \in C([0, T], \mathbb{R}): \| \Pi\|_{\infty} \leq M\}.
\]
{Then it is easy to verify that $K$ is a {nonempty},  {bounded},  {convex} and {{closed}}  subset of the Banach space \(C([0, T], \mathbb{R})\) with the uniform norm $\|\cdot\|_\infty$.} 
For any element $\Pi \in K$,  we  have the following: for any $t\in [0,T]$
\[
\begin{aligned}
 \|\mathcal{T}&(\Pi)(t)\|_\infty \leq   T(2|a| M  + \frac{b^2}{r} M^2 +q) \\  &
~ + q \int_t^T  \left|\exp\left[\bar{x}(s)+ \Phi^{-1}_{\alpha}  \bigg( \phi(s, 0)  V_0 +\sigma^2 \int_0^s \exp \left(\int_\tau^s 2\big(a - \frac{b^2}{r} \Pi_u \big) du \right) d\tau \bigg)\right]^{\frac12}\right| ds	\\
& \leq  T(2|a| M  + \frac{b^2}{r} M^2+ q)  \\&
~  +  Tq \exp\left[\mu^*+ |\Phi^{-1}_{\alpha}| \bigg( \exp  \left(2 T \big(|a| + \frac{b^2}{r} M \big)  \right)  V_0 +\sigma^2  T \exp \left(2 T \big(|a| + \frac{b^2}{r} M \big)  \right) \bigg)^{\frac12} \right]	\\
& = T\left[2|a| M  + \frac{b^2}{r} M^2+ q  +  q \exp\left[ \mu^* +   |\Phi^{-1}_{\alpha}| (V_0 + \sigma^2{T} )^{\frac12} \exp \left( T \big(|a| + \frac{b^2}{r} M \big)  \right)  \right] \right]		
\end{aligned}
\] 
with $\mu^*  : = \max\{\mu_0, e^{aT}\mu_0\}$. 
Therefore if we select the time $T$ small enough such that \eqref{eq:T-M relation} holds,
 we have 
\[
\|\mathcal{T}(\Pi)(t)\|_\infty \leq M, \quad \forall \Pi \in K,
\]
that is, after the mapping $\mathcal{T}$, the image $\mathcal{T}(\Pi)$ stays inside $K$. 

We have demonstrated that $\{\mathcal{T}(\Pi): \Pi \in K\}$  under the condition \eqref{eq:T-M relation} above is uniformly bounded. 
 Furthermore, one can show that $\{\mathcal{T}(\Pi): \Pi \in K\}$ is equicontinuous (see Lemma~\ref{lem:equicontinuouity}).
{Then by Arzela-Ascoli Theorem \cite[p.175]{conway2013course},  
the set $\{\mathcal{T}(\Pi): \Pi \in K\}$ is totally bounded and hence the closure of $\{\mathcal{T}(\Pi): \Pi \in K\}$ is compact \cite[Appendix A4]{rudin1976principles}. By definition \cite[Def. 4.1]{conway2013course}, 
  \(\mathcal{T}: K \to K \) is a compact operator.} 
 Since \( K \) be a closed, bounded, and convex subset of the normed space $C([0,T]; \mathbb{R})$ (endowed with the uniform norm) and   \( \mathcal{T}: K \to C([0,T]; \mathbb{R}) \) is a compact map such that \( \{\mathcal{T}(
 \Pi): \Pi \in K \}\subseteq K \) (under the condition \eqref{eq:T-M relation}), an application of the Schauder fixed-point theorem \cite[p.~150]{conway2013course} implies the equation \eqref{eq:Pi} has at least one a solution $\Pi \in K$. Then, given $\Pi \in K$, the equation \eqref{eq:dyn-Vmv} is a linear time-varying differential equation with continuous coefficients and hence the solution exists. Therefore
 	the coupled differential equations \eqref{eq:dyn-Pimv} and \eqref{eq:dyn-Vmv}  have at least one joint solution $(V, \Pi)$ with $\Pi \in K$ and $V \in C([0,T], \mathbb{R})$. 
 \end{proof}
 
 
 \begin{remark}
 	 We note that the Schauder fixed point theorem only allows us prove the existence of solutions, but the uniqueness is not guaranteed.  
 	 
 We use the Banach fixed point theorem to establish the existence of a unique solution over a small time horizon. 
  \end{remark}
  \begin{proposition}[Uniqueness]
 	Consider $[0,T]$ and $K  : = \{\Pi \in C([0, T], \mathbb{R}): \| \Pi\|_{\infty} \leq M\}$. If $T$ and $M$ satisfy  %
 	the following inequality %
 	\begin{equation} \label{eq:contraction-condition}
	T\bigg( 2|a| +\frac{b^2}{r} M^2 +q |\Phi^{-1}_{\alpha}| \frac{b^2}{r}T (V_0 + \sigma^2 T)^{\frac12} e^{\mu^*+ 3T(|a|+\frac{b^2}{r} M)+ |\Phi^{-1}_{\alpha}| [(V_0 + \sigma^2 T )^{\frac12} e^{T(|a|+ \frac{b^2}{r}M)}]}    \bigg)< 1,
\end{equation}
with $\mu^*  : = \max\{\mu_0, e^{aT}\mu_0\}$, 
 	 then    the coupled differential equations \eqref{eq:dyn-Pimv} and \eqref{eq:dyn-Vmv} with the coupling constraint \eqref{eq:Q-constraint}  have a unique  solution pair $(\Pi, V) \in K\times C([0,T], \mathbb{R})$. 
 \end{proposition} 
\begin{proof}
 For any two function $\Pi, ~\tilde{\Pi} \in K \subset C([0, T], 
 \mathbb{R})$, we have 
 \begin{equation*}
 	\begin{aligned}
 	|\mathcal{T}&(\Pi)(t) 	- \mathcal{T}(\tilde{\Pi})(t)| \\& \leq  \int_t^T \bigg|2a \Pi_s - \frac{b^2}{r} \Pi_s^2  		%
 		 + q +q\text{exp}\bigg [ \bar{x}(s)+   \Phi^{-1}_{\alpha}\big( \phi(s, 0)  V_0  + \sigma^2 \int_0^s \phi(s,\tau)  d\tau\big)^{\frac12}  \bigg] \\
 		 &\quad  
 		- \left(2a \tilde{\Pi}_s - \frac{b^2}{r} \tilde{\Pi}_s^2 ds
 		+ q +q\text{exp}\bigg [\bar{x}(s)+  \Phi^{-1}_{\alpha}\big( \tilde{\phi}(s, 0)  V_0  + \sigma^2 \int_0^t \tilde{\phi}(s,\tau)  d\tau\big)^{\frac12}  \bigg]\right) \bigg| ds\\
 		& =  \int_t^T \bigg|2a (\Pi_s-\tilde{\Pi}_s) - \frac{b^2}{r} (\Pi_s-\tilde{\Pi}_s)(\Pi_s+\tilde{\Pi}_s)  
 		+ q e^{\bar{x}(s)}  {\Delta }\bigg|ds
 	\end{aligned}
 \end{equation*}
 with 
 \[
 \begin{aligned}
 	 \Delta : = \text{exp}\bigg [ \Phi^{-1}_{\alpha}  \big( \phi(s, 0)  V_0  + \sigma^2 \int_0^s \phi(s,\tau)  d\tau\big)^{\frac12}  \bigg]~ - \text{exp}\bigg [ \Phi^{-1}_{\alpha} \big( \tilde{\phi}(s, 0)  V_0  + \sigma^2 \int_0^s \tilde{\phi}(s,\tau)  d\tau\big)^{\frac12}  \bigg]
 \end{aligned}
 \]
 and 
 \( 
 \tilde{\phi}(s,\tau) : = \text{exp} \left(\int_\tau^s 2(a-\frac{b^2}{r} \tilde{\Pi}_u)du \right). 
\)
In order to apply the Banach fixed point theorem, {we need to bound $\Delta$  in terms of $\|\Pi-\tilde{\Pi}\|_\infty$ to obtain a contraction condition.}
 We first need to establish the upper and lower bounds for the following term
\[
I_1(s) := (\Phi^{-1}_{\alpha})^2\big( \phi(s, 0)  V_0  + \sigma^2 \int_0^s \phi(s,\tau)  d\tau\big), \quad \forall s\in[0,T].
\] 
The upper bound is given by 
\begin{equation} \label{eq:phi-upper}
	\begin{aligned}
	 | I_1(s)| & = (\Phi^{-1}_{\alpha})^2\bigg | \text{exp} \left(\int_0^s 2(a-\frac{b^2}{r} \Pi_\tau)d\tau \right) V_0 + \sigma^2 \int_0^s \text{exp} \left(\int_\tau^s 2(a-\frac{b^2}{r} \Pi_q)dq \right) d\tau \bigg | \\
	& \leq (\Phi^{-1}_{\alpha})^2(V_0 + \sigma^2 T ) e^{(2T(|a|+ \frac{b^2}{r}M))}
\end{aligned}
\end{equation}
and  the lower bound by
\begin{equation} \label{eq:phi-lower}
	\begin{aligned}
	 I_1(s) %
	& \geq (\Phi^{-1}_{\alpha})^2 (V_0 + \sigma^2 T ) e^{(2T(-|a|- \frac{b^2}{r}M))} > 0.
\end{aligned}
\end{equation}
Furthermore, for all $s,\tau \in [0,T]$,
\begin{equation}\label{eq:phi-tildephi}
	\begin{aligned}
	(\Phi^{-1}_{\alpha})^2 |\phi(s,\tau)- \tilde{\phi}(s,\tau)| &= (\Phi^{-1}_{\alpha})^2\bigg|\text{exp} \left(\int_\tau^s 2(a-\frac{b^2}{r} \Pi_q)dq \right) - \text{exp} \left(\int_\tau^s 2(a-\frac{b^2}{r} \tilde{\Pi}_q)dq \right)\bigg|\\
	&  \leq (\Phi^{-1}_{\alpha})^2 e^{2T(|a|+\frac{b^2}{r} M)}\cdot \frac{{2} b^2}{r}T \cdot \|\Pi - \tilde{\Pi}\|_\infty 
\end{aligned}
\end{equation}
where the last step we use the fact that $|e^a - e^b|\leq e^{\max\{|a|,|b|\}}|a-b|$.
Hence Corollary \ref{cor:sqrt-bound-scale}, together with \eqref{eq:phi-lower} and \eqref{eq:phi-upper} and \eqref{eq:phi-tildephi} implies
\[
\begin{aligned}
	|\Delta|  &\leq  \frac{\text{exp}[|\Phi^{-1}_{\alpha}|(V_0 + \sigma^2 T )^{\frac12} e^{(T(|a|+ \frac{b^2}{r}M))}]}{2|\Phi^{-1}_{\alpha}|(V_0 + \sigma^2 T )^{\frac12} e^{(T(-|a|- \frac{b^2}{r}M))}} \\& \cdot (\Phi^{-1}_{\alpha})^2 \bigg| (\phi(s, 0) - \tilde{\phi}(s, 0)  ) V_0  + \sigma^2 \int_0^s (\phi(s,\tau)- \tilde{\phi}(s,\tau))  d\tau  \bigg| 
	\\& 
	\leq 
	\frac{\text{exp}[|\Phi^{-1}_{\alpha}|(V_0 + \sigma^2 T )^{\frac12} e^{T(|a|+ \frac{b^2}{r}M)}]}{2|\Phi^{-1}_{\alpha}|(V_0 + \sigma^2 T )^{\frac12} e^{T(-|a|- \frac{b^2}{r}M)}} \cdot  (\Phi^{-1}_{\alpha})^2 \cdot e^{2T(|a|+\frac{b^2}{r} M)}\cdot \frac{{2}b^2}{r}T\cdot (V_0 + \sigma^2 T) \cdot \|\Pi - \tilde{\Pi}\|_\infty \\
	& = e^{3T(|a|+\frac{b^2}{r} M)}\cdot e^{[|\Phi^{-1}_{\alpha}|(V_0 + \sigma^2 T )^{\frac12} e^{T(|a|+ \frac{b^2}{r}M)}]} \cdot  |\Phi^{-1}_{\alpha}| \cdot \frac{b^2}{r}T\cdot (V_0 + \sigma^2 T)^{\frac12} \cdot \|\Pi - \tilde{\Pi}\|_\infty   .
\end{aligned}
\]
Hence for all $t \in [0,T]$, the following holds
 \begin{equation}
 	\begin{aligned}
 		&|\mathcal{T}(\Pi)(t) - \mathcal{T}(\tilde{\Pi})(t)| \leq    \int_t^T \bigg|2a (\Pi_s-\tilde{\Pi}_s) - \frac{b^2}{r} (\Pi_s-\tilde{\Pi}_s)(\Pi_s+\tilde{\Pi}_s)  
 		+ q e^{\bar{x}(s)} \Delta \bigg|ds\\
 		& \leq T\bigg( 2|a| +\frac{b^2}{r} M^2  \\
 		& ~ +q  e^{\mu^* + 3T(|a|+\frac{b^2}{r} M)+ |\Phi^{-1}_{\alpha}|[(V_0 + \sigma^2 T )^{\frac12} e^{T(|a|+ \frac{b^2}{r}M)}]} \cdot  |\Phi^{-1}_{\alpha}| \cdot  \frac{b^2}{r}T (V_0 + \sigma^2 T)^{\frac12} \bigg)\|\Pi - \tilde{\Pi}\|_\infty 
 	\end{aligned}
 \end{equation}
 with $\mu^*  : = \max\{\mu_0, e^{aT}\mu_0\}$.
 Therefore 
 \begin{equation}
 	\begin{aligned}
 		\|&\mathcal{T} (\Pi) -\mathcal{T}(\tilde{\Pi})\|_\infty  \leq  T\bigg( 2|a| +\frac{b^2}{r} M^2 \\
 		&+q e^{\mu^*+ 3T(|a|+\frac{b^2}{r} M)+ |\Phi^{-1}_{\alpha}| [(V_0 + \sigma^2 T )^{\frac12} e^{T(|a|+ \frac{b^2}{r}M)}]} \cdot |\Phi^{-1}_{\alpha}| \cdot  \frac{b^2}{r}T (V_0 + \sigma^2 T)^{\frac12} \bigg)\|\Pi - \tilde{\Pi}\|_\infty  .
 	\end{aligned}
 \end{equation}
If we select $M$ and $T$ such that the inequality \eqref{eq:contraction-condition} holds,  
then the mapping $\mathcal{T}$ is a contraction under the sup norm  in $K$ over a small interval $[0, T]$. 
Furthermore, $K$ is complete under the sup norm  since $K$ is a closed subset of the complete metric space $C([0,T], R)$ under the sup norm (see e.g., \cite[p. 54]{rudin1976principles}). Then an  application of the Banach fixed point theorem implies that 
the equation \eqref{eq:Pi} has a unique solution $\Pi \in K$. This further implies that the joint solution pair $(\Pi, V)$ to \eqref{eq:dyn-Pimv} and \eqref{eq:dyn-Vmv} is unique with $\Pi \in K$ and $V\in C([0,T], \mathbb{R})$ and hence the proof is complete.
\end{proof}
 
\section{Numerical Illustrations} \label{sec:numerics}
Numerical results for solutions of the two-point boundary value problem specified by \eqref{eq:dyn-Pimv} and \eqref{eq:dyn-Vmv} with  the coupling \eqref{eq:Q-constraint} are demonstrated in Figure~\ref{fig:all_figures}.  Simulation results for implementing the MFG solution  are presented below in Figure~\ref{fig:cost-mf-comparision}, Figure~\ref{fig:cost} and Figure~\ref{fig:simulation-trajectory-20000}. 
\begin{figure}
    \centering
    \subfigure[$q_t^\alpha = q (1+e^{ \sqrt{V(t)} \Phi^{-1}_{\alpha}})\geq q\geq 0$]{
        \includegraphics[width=0.35\textwidth]{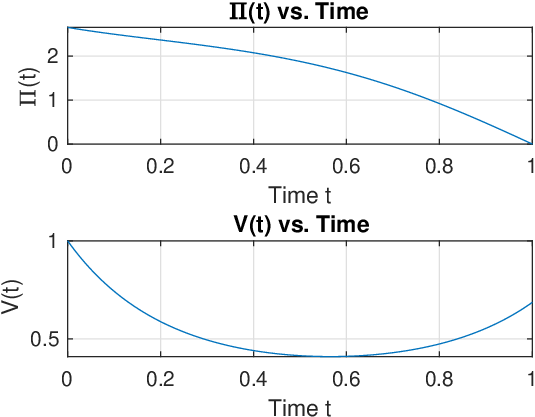}
        \label{fig:VPQeq}
    }
    \quad 
    \subfigure[$q_t^\alpha = q \geq 0$]{
 \includegraphics[width=0.35\textwidth]{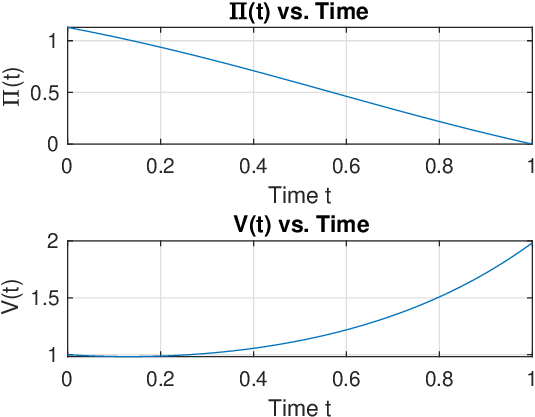}
        \label{fig:figure2}
    }
    \caption{Trajectories of $V$ and $\Pi$.  We see a clear reduction of the variance $V$ in (a) compared to that in (b), and $\Pi$ is larger in (a) than in (b).  The parameters are $a = 0.5$, $\mu_0=0$, $b=r=1$, $\sigma =V_0=q=1$, $\alpha = 0.95$, and $T=1$. In this case, the inequality \eqref{eq:T-M relation} is not satisfied, but the numerical solution to the two-point boundary value problem (specified by \eqref{eq:dyn-Pimv} and \eqref{eq:dyn-Vmv} with the coupling \eqref{eq:Q-constraint}) still exists as the inequality in \eqref{eq:T-M relation} is only a sufficient condition for the existence.   }
    \label{fig:all_figures}
\end{figure}

\begin{figure}
\subfigure[]{
	\includegraphics[width=6cm]{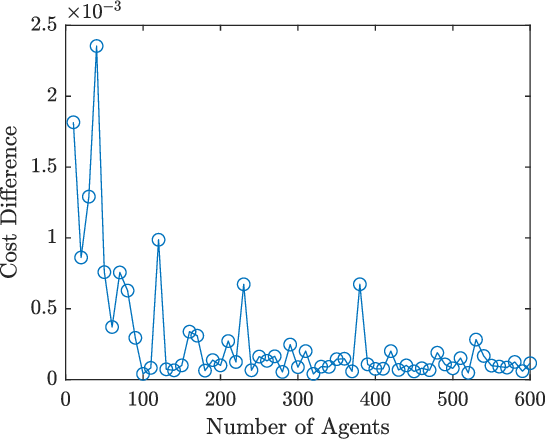}} 
\subfigure[]	{ 
	\includegraphics[width=6cm]{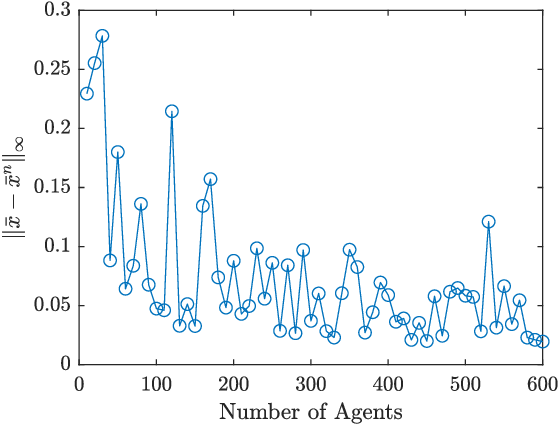}}\caption{The difference between the cost of the MFG tracking control and that of the optimal tracking control (assuming the data of all other agents are given a priori), as the number of agents increases, is illustrated in (a). The maximum difference (over time) between the limit population mean  and the mean of the finite population, as the number of agents increases, is illustrated in  (b). The parameters are the same as those in Figure~\ref{fig:simulation-trajectory-20000}.}\label{fig:cost-mf-comparision}
\end{figure}

\begin{figure}
	\includegraphics[width =8cm]{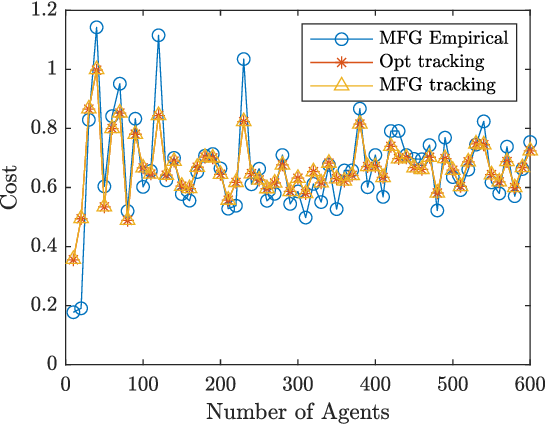}
	\caption{The agent cost of the optimal tracking solution (assuming the state trajectories of other agents are given a priori) is illustrated in star points, and the agent cost using the MFG solution (assuming the state trajectories of other agents are given a priori) is illustrated in triangle points, and finally the agent cost (by computing empirical average of cost over simulations) of using the MFG solution is illustrated in circle points. All the simulations use the same parameters as those in Figure~\ref{fig:simulation-trajectory-20000}.}\label{fig:cost}
\end{figure}

\begin{figure}
		\includegraphics[width=1\textwidth]{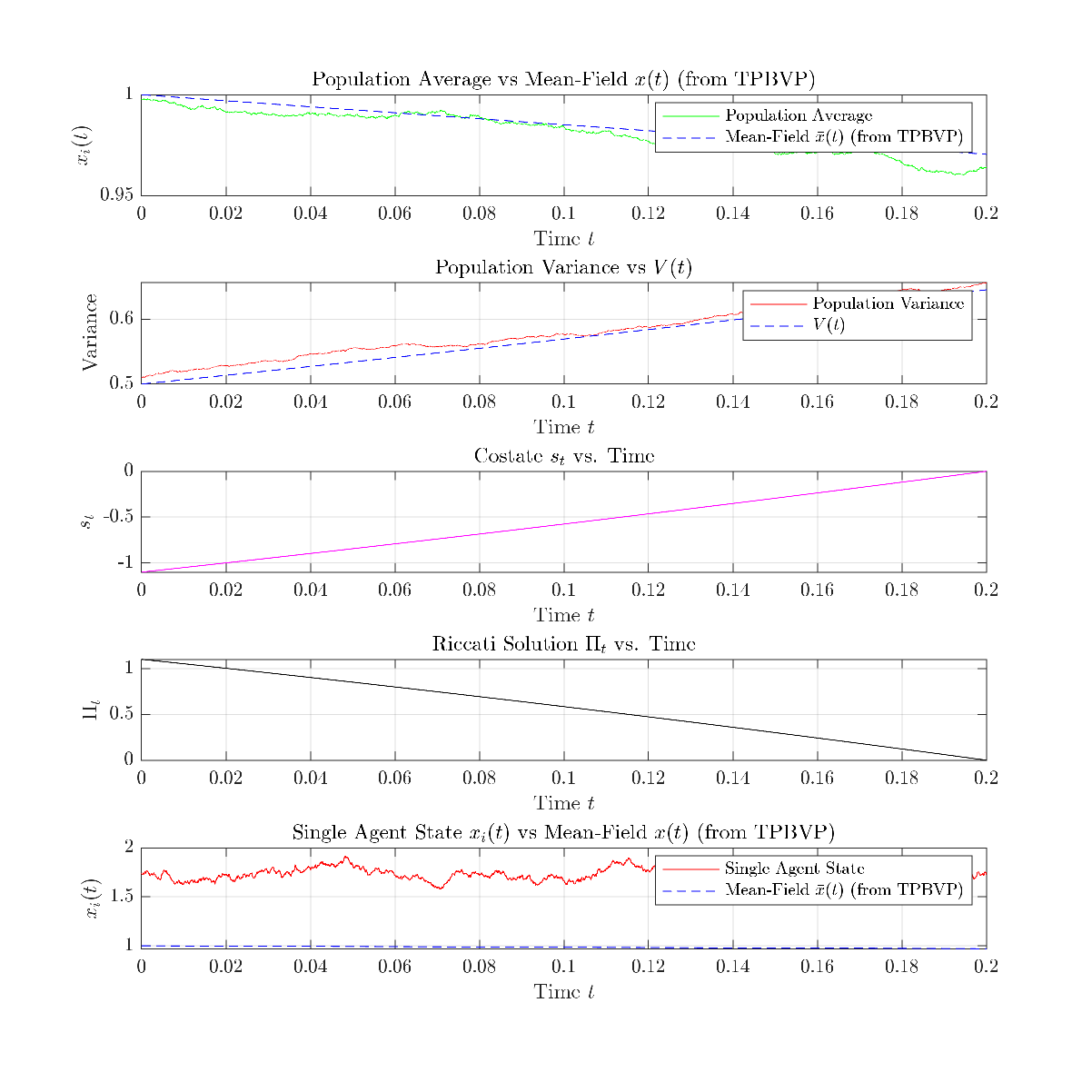}
		\caption{Simulation results with  2000 agents using MFG solutions. The parameters in the simulations are
	\( a = -0.15 \), \( b = 0.75 \), \( r = 3.5 \), \( \sigma = 1 \), \( V_0 = 0.5 \), \( \alpha = 0.975 \), \( q = 0.45 \), \( T = 0.2 \) and \( \mu_0 = 1.0 \).
 In this case, both inequalities  \eqref{eq:T-M relation} and \eqref{eq:contraction-condition} are satisfied for \( M = 3 \).} \label{fig:simulation-trajectory-20000}
\end{figure}

\section{Conclusion}\label{sec:conclusion}
This paper presents an initial investigation of the linear quadratic MFG problems with the cost coefficients that depend on the quantile of the mean field distribution. Such problems can be transformed into linear quadratic MFG problems with mean-variance dependent cost coefficients. The corresponding solution involves two coupled differential equations:  one Riccati equation and one variance evolution equation. Conditions under which the solution exists and is unique are established.  

 Further investigation should include  the following directions:  
(a) similar formulations where each agent has a vector state, (b) general forms of MFGs with mean-variance dependence, (c) $\varepsilon$-Nash properties for such problems, (d)  the case with mean-field couplings in the agent dynamics, (e)  similar problems where the model parameters are non-uniform and follow certain distributions, and (f) the extensions of this work to network-coupled mean field systems. 

\section{Appendix}
 \begin{lemma} \label{lem:sqrt-bound} For $a>0$ and $b>0$, the following inequality holds
  	\[
 |e^{\sqrt{a}} -e^{\sqrt{b}}| \leq \frac{e^{{\max\{\sqrt{|a|}, \sqrt{|b|}\}}}}{2\min\{\sqrt{a}, \sqrt{b}\}} \cdot |a-b|.
 \]
  \end{lemma}
 \begin{proof}
 	By the mean value theorem, the exists $c \in (x, y)$ (assuming $x >y$ without loss of generality) such that
$
e^{c} = \frac{e^{x}-e^y}{x-y}
$
and this implies
$
|e^{\sqrt{a}} - e^{\sqrt{b}}|\leq e^{\max\{|\sqrt{a}|, |\sqrt{b}|\}}|\sqrt{a}-\sqrt{b}|.
$
We note that 
$
\sqrt{a} -\sqrt{b} = \frac{1}{\sqrt{a} +\sqrt{b}} (a-b),
$
and hence 
$
|\sqrt{a} -\sqrt{b}| = \frac{1}{\sqrt{a} +\sqrt{b}} |a-b| \leq \frac{1}{2 M_0} |a-b|
$
where $M_0 : =\min\{\sqrt{a}, \sqrt{b}\}>0 $. 
  Combining the observations above leads to the desired result.
 \end{proof}
 
 An immediate consequence of Lemma \ref{lem:sqrt-bound} is the following.
\begin{corollary}\label{cor:sqrt-bound-scale} For $a>0$, $b>0$ and $c>0$, the following inequality holds
  	\[
 |e^{c\sqrt{a}} -e^{c\sqrt{b}}| \leq \frac{e^{{\max\{c\sqrt{|a|}, c\sqrt{|b|}\}}}}{2\min\{c\sqrt{a}, c\sqrt{b}\}} \cdot c^2|a-b|.
 \]
\end{corollary}

\begin{lemma}[Equicontinuity]\label{lem:equicontinuouity}
	Let $\mathcal{T}$ be given by \eqref{eq:Tao-Mapping} and $K  := \{\Pi \in C([0, T], \mathbb{R}): \| \Pi\|_{\infty} \leq M\}$ for some $M>0$. Then the set $\{\mathcal{T}(\Pi): \Pi \in K\}$ is equicontinuous. 
\end{lemma}
\begin{proof}
To show the equicontinuity \cite[Def. 7.22]{rudin1976principles} of $\{\mathcal{T}(\Pi): \Pi \in K\}$ with $\mathcal{T}$ given by \eqref{eq:Tao-Mapping},
 we need to prove that for every \(\epsilon > 0\), there exists a \(\delta > 0\) such that for any \(t_1, t_2 \in [0, T]\) with \(|t_1 - t_2| < \delta\), the following inequality holds 
$
|\mathcal{T}(\Pi)(t_1) - \mathcal{T}(\Pi)(t_2)| < \epsilon$  for all $\Pi  \in K.
$
Consider the difference:
\[
\begin{aligned}
|\mathcal{T}(\Pi)(t_1) - \mathcal{T}(\Pi)(t_2)| &= \left| \int_{t_1}^{t_2} {I}(s) \, ds \right| \leq  \int_{t_1}^{t_2} \left|{I}(s)  \right|\, ds
\end{aligned}
\]
where \({I}(s)\) is the integrand for \(\mathcal{T}(\Pi)(t)\) given by
\[
\begin{aligned}
	&I(s)  = 2a \Pi_s - \frac{b^2}{r} \Pi_s^2 + q +  q\exp\left[\bar{x}(s)+\Phi^{-1}_{\alpha}  F ^{\frac12} \right].
\end{aligned}
\]
where 
$
F := \exp\left(\int_0^s 2\left(a-\frac{b^2}{r} \Pi_u\right)du\right)  V_0 + \sigma^2 \int_0^s \exp \left(\int_\tau^s 2\left(a - \frac{b^2}{r} \Pi_u \right) du \right) d\tau.
$
Since  \(\|\Pi\|_\infty \leq M\) for a given \(M > 0\), the terms inside the integrand $I(s)$ are uniformly bounded: first,
\[
|2a \Pi_s - \frac{b^2}{r} \Pi_s^2 + q | \leq 2|a|M +  \frac{b^2}{r}M^2 + q;
\]
second,  the exponential term inside $I(s)$ that involves \(\Pi\) is also uniformly bounded as follows:
\[
\begin{aligned}
q\exp\left[\Phi^{-1}_{\alpha}  F ^{\frac12} \right]	\leq q\exp\left[|\Phi^{-1}_{\alpha}| \cdot  (V_0 + \sigma^2 T )^{\frac12} e^{(T(|a|+ \frac{b^2}{r}M))} \right] 
\end{aligned}
\]
since
$
	|F| 
	\leq (V_0 + \sigma^2 T ) e^{(2T(|a|+ \frac{b^2}{r}M))}. 
$
Let the uniform bound be denoted by \[C := 2|a|M +  \frac{b^2}{r}M^2 + q +  q\exp\left[\mu^*+ |\Phi^{-1}_{\alpha}| \cdot  (V_0 + \sigma^2 T )^{\frac12} e^{(T(|a|+ \frac{b^2}{r}M))} \right]  > 0\] with $\mu^* := \max\{\mu_0, e^{aT}\mu_0\}$. Then we have
\[
\begin{aligned}
|\mathcal{T}(\Pi)(t_1) - \mathcal{T}(\Pi)(t_2)|  
\leq \int_{t_1}^{t_2} |I(s)| \, ds 
\leq \int_{t_1}^{t_2} C \, ds = C |t_2 - t_1|.
\end{aligned}
\]
Hence, if we choose \(\delta = \frac{\epsilon}{C}\), then 
$
|\mathcal{T}(\Pi)(t_1) - \mathcal{T}(\Pi)(t_2)| < \epsilon
$
for all \(\Pi \in K\) whenever \(|t_1 - t_2| < \delta\). That is, the set $\{\mathcal{T}(\Pi): \Pi \in K\}$ is equicontinuous. 
\end{proof}
\bibliographystyle{plain}
\bibliography{mybib}
\end{document}